\documentclass[12pt]{article}
\usepackage{amsthm}
\usepackage{amsmath}
\usepackage{amsfonts}
\usepackage{graphicx}
\usepackage{latexsym}
\usepackage{amssymb}

\newcommand{\deff}{\mbox{$\stackrel{\rm def}{=}$}}

\newcommand{\field}[1]{\mathbb{#1}}

\newcommand{\Z}{\field{Z}}

\newcommand{\F}{\field{F}}

\newcommand{\dS}{\field{S}}

\newcommand{\cP}{{\cal P}}

\newcommand{\cG}{{\cal G}}

\newtheorem{theorem}{Theorem}
\newtheorem{lemma}{Lemma}

\newtheorem{cor}{Corollary}

\begin{document}

\bibliographystyle{plain}

\title{
\begin{center}
Automorphisms of Codes in the Grassmann Scheme
\end{center}
}
\author{
{\sc Tuvi Etzion}\thanks{Department of Computer Science, Technion,
Haifa 32000, Israel, e-mail: {\tt etzion@cs.technion.ac.il}.} \and
{\sc Alexander Vardy}\thanks{Department of Electrical Engineering,
University of California San Diego, La Jolla, CA 92093, USA,
e-mail: {\tt avardy@ucsd.edu}.}}

\maketitle

\begin{abstract}
Two mappings in a finite field, the Frobenius mapping
and the cyclic shift mapping, are applied
on lines in PG($n,p$) or codes in the Grassmannian,
to form automorphisms groups in the Grassmanian and in its codes.
These automorphisms are examined on two classical
coding problems in the Grassmannian. The first is
the existence of a parallelism with lines in the related
projective geometry and the second
is the existence of a Steiner structure.
A computer search was applied to find
parallelisms and codes. A new parallelism of lines in PG(5,3) was
formed. A parallelism with these parameters was not known before.
A large code which is only slightly short of a Steiner structure was
formed.
\end{abstract}

\vspace{0.5cm}


%\footnotetext[1] { This research was supported in part by the
%United States --- Israel Binational Science~Foundation (BSF),
%Jerusalem, Israel, under Grant 2006097.  }

%%%%%%%%%%%%%%%%%%%%%%%%%%%%%%%%%%%%%%%%%%%%%%%%%%%%%%%%%%%%%%%%%%%%%%
%%%%%%%%%%%%%%%%%%%%%%%%%%%%%%%%%%%%%%%%%%%%%%%%%%%%%%%%%%%%%%%%%%%%%%
%%%%%%%%%%%%%%%%%%%%%%%%%%%%%%%%%%%%%%%%%%%%%%%%%%%%%%%%%%%%%%%%%%%%%%
\newpage
\section{Introduction}
\label{sec:introduction}

The Grassmannian $\cG_q(n,k)$ is the set of all
$k$-dimensional subspaces of an $n$-dimensional
subspace over the finite field $\F_q$. A code
in $\cG_q(n,k)$ is a subset
of $\cG_q(n,k)$. There has been lot of interest in
these codes in the last five year due to their
application in network coding~\cite{KoKs}.
Our motivation for this work also came from this application in
network coding.

Some of the coding problems in the Grassmannian
were formulated in the past in terms of projective geometry
or $q$-analog of block design. In this paper we will
consider two of these problems. The first is
the existence of a parallelism with lines in PG($n,p$)
and the second is the existence of a Steiner structure.

Steiner structures are known also as $q$-analog of Steiner systems.
A \emph{Steiner structure} $\dS_q[t,k,n]$ is a set $\dS$
of $k$-dimensional subspaces of~$\F_q^n$ such that
each $t$-dimensional subspace of $\F_q^n$ is contained in exactly one
subspace of $\dS$. Steiner structures were considered in many
papers~\cite{AAK,EtVa11,ScEt02,Tho1,Tho2}, where they have other names as well.
An $\dS_q[t,k,n]$ can be readily constructed for $t=k$
and for $k=n$. If $t=1$ these structures are called
$k$-\emph{spreads} and they are known to exist if and only
if $k$ divides~$n$. These structures are also considered to be
trivial. The first nontrivial case is a Steiner structure
$\dS_2[2,3,7]$. The possible existence of this structure was considered
by several authors, and some conjectured~\cite{Met99} that it doesn't exist
and that generally nontrivial
Steiner structures do not exist.

$k$-spreads were considered in numerous papers. These
can be viewed as partition of the points set of PG($n-1,q$) into
disjoint $(k-1)$-dimensional subspaces (in the geometry). It is
called a $(k-1)$-spread in the geometry. Two disjoint spreads
are called \emph{parallel spreads} and a partition of the
$\frac{q^n-1}{q-1}$ points of PG($n-1,q$) into disjoint
$(k-1)$-spreads is called a \emph{parallelism}. The only known parallelism
with $(k-1)$-spreads, $k>2$, is for $k=3$, $q=2$,
and $n=6$~\cite{Sar02}.
It is known that parallelisms with 1-spreads exist for $q=2$ and
all even $n$~\cite{Bak76,ZZS71}, and for each prime power~$q$,
where $n=2^m$~\cite{Beu74}.
No other parallelisms are known.

In this paper we consider solutions for these two problems,
the existence of parallelisms
with 1-spreads in PG($n-1,p$) (which are 2-spreads in
$\cG_p(n,2)$) and the existence of nontrivial Steiner structures.
We will use two types of mappings, the Frobenius mappings
and cyclic shift mappings to form an automorphisms group
in codes and as a mechanism
to obtain disjoint spreads and Steiner structures.
We examine 1-spreads in which the Frobenius mappings form
an automorphism group and a parallelism is obtained by using
cyclic shifts on two disjoint cycles of PG($n-1,p$).
The method was found to be successful for $p=3$ and $n=6$.
We conjecture that the method would be successful whenever
some necessary conditions are satisfied.
To form a Steiner structure
$\dS_2 [2,3,13]$ fifteen representatives are needed.
It turns out that fourteen such representatives are easy to obtain
and it is an open problem whether fifteen representatives
exist.

The rest of this paper is organized as follows.
In Section~\ref{sec:map} we define the two types of mappings
and state some of their properties.
In Section~\ref{sec:parallel} we use the two types of mappings for the
construction of a parallelism of lines in PG($n,p$).
In Section~\ref{sec:Steiner}
we use the two types of mappings for an attempt to construct
nontrivial Steiner structures.
Conclusion and problems for further research are given
in Section~\ref{sec:conclusion}.

\section{Mappings in the Grassmannian}
\label{sec:map}

Let $\F_{p^n}$ be a finite field with $p^n$ elements, where $p$ and $n$ are primes,
and let $\alpha$ be a primitive element in $\F_{p^n}$.

The \emph{Frobenius mapping} $\Upsilon_\ell$,
$0 \leq \ell \leq n-1$, $\Upsilon_\ell : \F_{p^n} \setminus \{ 0 \}
\longrightarrow \F_{p^n} \setminus \{ 0 \}$
is defined by $\Upsilon_\ell (x) \deff x^{p^\ell}$ for each
$x \in \F_{p^n} \setminus \{ 0 \}$.

The \emph{cyclic shift mapping} $\Phi_j$, $0 \leq j \leq p^n-2$,
$\Phi_j : \F_{p^n} \setminus \{ 0 \}
\longrightarrow \F_{p^n} \setminus \{ 0 \}$ is defined by
$\Phi_j (\alpha^i) \deff \alpha^{i+j}$, for each $0 \leq i \leq p^n-2$.

The two types of mappings $\Upsilon_\ell$ and $\Phi_j$ can be applied
on a subset or a subspace, by applying the mapping on each
element of the subset or subspace, respectively. Formally,
given two integers $0 \leq \ell \leq n-1$ and $0 \leq j \leq p^n-2$,
$$
\Upsilon_\ell \{ x_1 , x_2 , \ldots , x_r \} \deff
\{ \Upsilon_\ell (x_1), \Upsilon_\ell (x_2),\ldots,\Upsilon_\ell(x_r) \}~,
$$
$$
\Phi_j \{ x_1 , x_2 , \ldots , x_r \} \deff \{ \Phi_j (x_1), \Phi_j
(x_2),\ldots,\Phi_j (x_r) \}.
$$

\begin{lemma}
\label{lem:inverse}
The mappings $\Upsilon_\ell$ and $\Phi_j$ are invertible.
\end{lemma}
\begin{proof}
Clearly, $\Upsilon_\ell^{-1} = \Upsilon_{n-\ell}$ and
$\Phi_j^{-1} = \Phi_{2^n-1-j}$.
\end{proof}

For a given integer $s \in \Z_{p^n-1}$,
the cyclotomic coset $C_s$ is defined by
$$
C_s \deff \{ s \cdot p^i ~:~ 0 \leq i \leq n-1 \}~.
$$
The smallest element in a cyclotomic coset is called
the \emph{coset representative}. Let $C(s)$ denote the
coset representative of $C_s$, i.e. if $r$ is the coset
representative for the coset of $s$, $C_s$,
then $r=C(s)$. The first lemma is well known and
easy to verify.

\begin{lemma}
The size of a cyclotomic coset is either $n$ or \emph{one}. There are
exactly $\frac{p^n-p}{n}$ different cyclotomic cosets of size $n$.
\end{lemma}

The definitions of the Frobenius mappings and the cyclotomic cosets
imply the following lemma.

\begin{lemma}
When applied on $\F_{p^n} \setminus \{ 0 \}$ the
Frobenius mappings forms an equivalence relation on
$\F_{p^n} \setminus \{ 0 \}$,
where  an equivalence class contains the powers of $\alpha$
which are exactly the elements of one cyclotomic cosets of
$\Z_{p^n-1}$.
\end{lemma}

Another well known result is the following lemma.

\begin{lemma}
\label{lem:isomorphism}
The finite field $\F_{p^m}$ and the vector space $\F_p^n$
are isomorphic.
\end{lemma}

In view of Lemma~\ref{lem:isomorphism} we can apply the Frobenius mappings and
the cyclic shifts mapping on $\F_p^m$ exactly as they are applies
on $\F_{p^m}$. If $h: \F_{p^m} \longrightarrow \F_p^n$ is the
isomorphism from $\F_{p^m}$ to $\F_p^m$ such that $y=h(x)$ for
$y \in \F_p^m$ and $x \in \F_{p^m}$ then
$$
\Upsilon_\ell (y) \deff h(\Upsilon_\ell (x)) ~~ \text{and} ~~ \Phi_j (y) \deff h(\Phi_j (x)) ~,
$$
for every $0 \leq \ell \leq n-1$ and $0 \leq j \leq p^n-2$.

\begin{lemma}
Let $X$ and $Y$ be two $k$-dimensional subspaces of $\F_p^n$
such that there exist two integers, $\ell_1$, $0 \leq \ell_1 \leq n-1$,
and $j_1$, $0 \leq j_1 \leq p^n-1$, such that $Y= \Phi_{j_1} (\Upsilon_{\ell_1} (X))$.
Then there exist two integers, $\ell_2$, $0 \leq \ell_2 \leq n-1$,
and $j_2$, $0 \leq j_2 \leq p^n-1$, such that $Y= \Upsilon_{\ell_2} ( \Phi_{j_2} (X))$.
\end{lemma}
\begin{cor}
The combination of the Frobenius mappings and the cyclic shifts mappings
induce an equivalence relation of the set of all $k$-dimensional subspaces
of $\F_p^n$.
\end{cor}

%Let $\alpha$ be a primitive element is $\F_{2^n}$.
%We say that a code $\C \subseteq \cG_2 (n,k)$ is
%\emph{cyclic} if it has the following property:
%whenever $\{ {\bf 0}, \alpha^{i_1},\alpha^{i_2},\ldots,
%\alpha^{i_m} \}$ is a codeword of $\C$, so is its
%cyclic shift $\{ {\bf 0}, \alpha^{i_1+1},\alpha^{i_2+1},\ldots,
%\alpha^{i_m+1} \}$. In other words, if we map each vector
%space $V \in \C$ into the corresponding binary
%\emph{characteristic vector} $x_V = (x_0,x_1,\ldots,x_{2^n-2})$
%given by
%$$
%x_i=1 ~ \text{if}~ \alpha^i \in V~ \text{and} ~ x_i =0~\text{if}~\alpha^i \not\in V
%$$
%then the set of all such characteristic vectors is
%closed under cyclic shifts. Note that the property
%of being cyclic does not depend on the choice of the
%primitive element $\alpha \in \F_{2^n}$.

\section{Parallelism of lines in PG($n,p$)}
\label{sec:parallel}

Let $\alpha$ be a primitive element in $\F_{p^n}$, where
$p$ and $n$ are primes. Let
$\cP_0 = \{ (\alpha^i,0) ~:~ 0 \leq i \leq \frac{p^n-1}{p-1} -1 \}$
and $\cP_1 = \{ (x,1) ~:~ x \in \F_{p^n} \}$.
The points of the projective geometry PG($n,p$) are
represented by $\cP_0 \cup \cP_1$.

The Frobenius mapping $\Upsilon_\ell$ is applied on a subset
$X \in \cP_0 \cup \cP_1$ as follows. For a given point
$( \alpha^i ,0) \in \cP_0$, and an integer $\ell$, we define $\Upsilon_\ell (( \alpha^i ,0))
\deff (\alpha^{p^\ell \cdot i},0)$, where the exponent is taken modulo
$\frac{p^n-1}{p-1}$. For a given point $(y,1) \in \cP_1$, we have
$y = \alpha^i$, for some $0 \leq i \leq p^n-2$. We define
$\Upsilon_\ell ((y,1)) \deff ( \alpha^{p^\ell \cdot i} ,1)$,
where the exponent is taken modulo $p^n-1$;
if $y=0$ then we define $\Upsilon_\ell ((0,1)) \deff (0,1)$. If
$X = \{ x_1 , x_2 , \ldots , x_m \}$ then
$\Upsilon_\ell (X) \deff \{ \Upsilon_\ell (x_1) , \Upsilon_\ell (x_2), \ldots ,
\Upsilon_\ell (x_m) \}$.

Similarly, the cyclic shift mapping is applied on a subset
$X \in \cP_0 \cup \cP_1$. For a given point
$( \alpha^i ,0) \in \cP_0$, and an integer $j$, we defined $\Phi_j (( \alpha^i ,0))
\deff (\alpha^{i+j},0)$, where the exponent is taken modulo
$\frac{p^n-1}{p-1}$. For a given point $(y,1) \in \cP_1$,
we have $y = \alpha^i$, for some $0 \leq i \leq p^n-2$. We define
$\Phi_j ((y,1)) \deff ( \alpha^{i+j} ,1)$, where
the exponent is taken modulo $p^n-1$;
if $y=0$ then we define $\Phi_j ((0,1)) \deff (0,1)$. If
$X = \{ x_1 , x_2 , \ldots , x_m \}$ then we define
$\Phi_j (X) \deff \{ \Phi_j (x_1) , \Phi_j (x_2), \ldots ,
\Phi_j (x_m) \}$.

The results on the two types
of mappings given in Section~\ref{sec:map}
hold also for this representation of PG($n,p$) and these
modified mappings.
In particular each one of the two types of mappings and
the combination of the two induce an equivalence
relation on the set of all $k$-dimensional subspaces of $\F_p^n$,
also for this representation of PG($n,p$).

A line in PG($n,p$) consists of $p+1$ points.
Either all the $p+1$ points are contained in $\cP_0$
or exactly one point is contained in $\cP_0$ and
$p$ points are contained in $\cP_1$. A 1-spread (or spread in short)
contains $p^{n-1} +p^{n-3} + \cdots + p^2 +1$ lines,
from which exactly $p^{n-1}$ contain points from $\cP_1$.

We construct the first spread as follows. The spread contains
three types of lines.
\begin{enumerate}
\item
The first type contains only one line. This line contains the
points $(\alpha^0 , 0)$, $(0,1)$, and $(\alpha^{{\frac{p^n-1}{p-1}}i},1)$,
for all $0 \leq i \leq p-2$.

\item
The second type contains $p^{n-1}-1$ lines, each line
has exactly one point from $\cP_0$ and~$p$ points from $\cP_1$.
$X$ is a line of this type if and only if $\Upsilon_\ell (X)$
is a line of this type, for every $\ell$, $0 \leq \ell \leq n-1$.
The lines of this type are partitioned
into $p-1$ sets. Each set has $\frac{p^{n-1}-1}{p-1}$ lines.
For each line $X$ in the $i$-th set there is a line $Y$
in the $(i+1)$-th set such that $Y = \Phi_{(p-1)j+1} (X)$
for some integer $j$.
For every two distinct given lines of this type, $X$ and $Y$,
there is no $j$ such that $Y = \Phi_{(p-1)j} (X)$.
Therefore, we can consider for this type
$\frac{p^{n-1}-1}{(p-1)n}$ lines as base lines. Each such line
is shifted $p-1$ shifts to form base lines for
each set, one for each residue modulo $p-1$.
On these shifts we apply the Frobenius mappings. Note, that
since $t \cdot p \equiv t~(\bmod~p-1)$ the Frobenius mappings
preserve the set in which the subspace of this type is contained.

\item
The third type contains
$\frac{p^{n-1}-1}{p^2-1}=p^{n-3} +p^{n-5}+ \cdots + p^2 +1$
lines, where
all the points of each line are contained in $\cP_0$.
$X$ is a line of this type if and only if $\Upsilon_1 (X)$
is a line of this type. For each two given distinct lines
of this type, $X$ and $Y$,
there is no $j$ such that $Y = \Phi_j (X)$. Therefore,
we can view $\frac{p^{n-3} +p^{n-5}+ \cdots + p^2 +1}{n}$
of these lines as generator lines. The Frobenius mappings are
applied on these generator lines.
\end{enumerate}

The base line and the generator lines should have some difference
properties. We skip the description of these difference properties
and just mention that they are similar to the ones described in Section~\ref{sec:Steiner}.

The $(i+1)$-th spread, $\dS_{i+1}$, is generated from the $i$-th spread, $\dS_i$,
as follows. If $X$ is a line of~$\dS_i$ then $\Phi_{p-1} (X)$ is a line
in $\dS_{i+1}$.
In this way we generate a total of $\frac{p^n-1}{p-1}$
pairwise disjoint spreads.
It can be verified that if we want to
obtain distinct lines in the
$\frac{p^n-1}{p-1}$ spreads then $p-1$ must be relatively prime
to $\frac{p^n-1}{p-1}$.

Clearly, our description of the construction implies some divisibility
conditions. Some of these conditions always hold and only the necessary
proof should be given.
The conditions which are not always satisfied are
summarized in the following lemma.

\begin{lemma}
\label{lem:condS}
For the construction of $\frac{p^n-1}{p-1}$ pairwise disjoint
spreads, given in this section, the following conditions must hold.
\begin{enumerate}
\item $n$ is odd.

\item $n$ divides $\frac{p^{n-1}-1}{p^2-1}$.

\item $\gcd (p-1, \frac{p^n-1}{p-1})=1$.
\end{enumerate}
\end{lemma}

We conjecture that
if the conditions of Lemma~\ref{lem:condS} are satisfied
then there is a parallelism of lines in PG($n,p$) obtained
by the method given in this section.

We demonstrate the construction to one of a numerous
number of solutions for $p=3$ and
$n=5$. Let $\alpha$ be a root
of the primitive polynomial $x^5+2x^4+1$ over $\F_3$.
There are 8 base lines of the second type.
They are shifted by  1, 5, 29, 111, 61, 187, 129, and 125, respectively for
the odd shifts, and by 218, 8, 12, 230, 150, 132, 202, 40, respectively
for the even shifts.
The outcome are the sixteen lines given in the following table.

\begin{table}[h]
\centering
\begin{scriptsize}
\begin{tabular}{|c|c|c|c|}
\hline $~$ &   base line & odd shift &even shift
\tabularnewline \hline \hline
& & & \\
1 &  $\{ (\alpha^5,0), (\alpha^0,1) , (\alpha^1,1) , (\alpha^{190},1) \}$ &
 $\{ (\alpha^6,0), (\alpha^1,1) , (\alpha^2,1) , (\alpha^{191},1) \}$ &
 $\{ (\alpha^{102},0), (\alpha^{218},1) , (\alpha^{219},1) , (\alpha^{166},1) \}$
\tabularnewline \hline
& & & \\
2 &  $\{ (\alpha^{74},0), (\alpha^0,1) , (\alpha^2,1) , (\alpha^{167},1) \}$ &
 $\{ (\alpha^{79},0), (\alpha^5,1) , (\alpha^7,1) , (\alpha^{172},1) \}$ &
 $\{ (\alpha^{82},0), (\alpha^8,1) , (\alpha^{10},1) , (\alpha^{175},1) \}$
\tabularnewline \hline
& & & \\
3 &  $\{ (\alpha^{120},0), (\alpha^0,1) , (\alpha^4,1) , (\alpha^{68},1) \}$ &
 $\{ (\alpha^{28},0), (\alpha^{29},1) , (\alpha^{33},1) , (\alpha^{97},1) \}$ &
 $\{ (\alpha^{11},0), (\alpha^{12},1) , (\alpha^{16},1) , (\alpha^{80},1) \}$
 \tabularnewline \hline
& & & \\
4 &  $\{ (\alpha^{69},0), (\alpha^0,1) , (\alpha^5,1) , (\alpha^{122},1) \}$ &
 $\{ (\alpha^{59},0), (\alpha^{111},1) , (\alpha^{116},1) , (\alpha^{233},1) \}$ &
 $\{ (\alpha^{57},0), (\alpha^{230},1) , (\alpha^{235},1) , (\alpha^{110},1) \}$
 \tabularnewline \hline
& & & \\
5 & $\{ (\alpha^{67},0), (\alpha^0,1) , (\alpha^8,1) , (\alpha^{220},1) \}$ &
 $\{ (\alpha^7,0), (\alpha^{61},1) , (\alpha^{69},1) , (\alpha^{39},1) \}$ &
 $\{ (\alpha^{96},0), (\alpha^{150},1) , (\alpha^{158},1) , (\alpha^{128},1) \}$
 \tabularnewline \hline
& & & \\
6 &  $\{ (\alpha^{27},0), (\alpha^0,1) , (\alpha^{14},1) , (\alpha^{169},1) \}$ &
 $\{ (\alpha^{93},0), (\alpha^{187},1) , (\alpha^{201},1) , (\alpha^{114},1) \}$ &
 $\{ (\alpha^{38},0), (\alpha^{132},1) , (\alpha^{146},1) , (\alpha^{59},1) \}$
 \tabularnewline \hline
& & & \\
7 & $\{ (\alpha^{37},0), (\alpha^0,1) , (\alpha^{20},1) , (\alpha^{147},1) \}$ &
 $\{ (\alpha^{45},0), (\alpha^{129},1) , (\alpha^{149},1) , (\alpha^{34},1) \}$ &
 $\{ (\alpha^{118},0), (\alpha^{202},1) , (\alpha^{222},1) , (\alpha^{107},1) \}$
 \tabularnewline \hline
& & & \\
8 & $\{ (\alpha^{97},0), (\alpha^0,1) , (\alpha^{31},1) , (\alpha^{177},1) \}$ &
 $\{ (\alpha^{101},0), (\alpha^{125},1) , (\alpha^{156},1) , (\alpha^{60},1) \}$ &
 $\{ (\alpha^{16},0), (\alpha^{40},1) , (\alpha^{71},1) , (\alpha^{217},1) \}$
 \tabularnewline \hline

\end{tabular}
\end{scriptsize}
\end{table}
The generator lines of the third type are
$$\{ (\alpha^8,0),(\alpha^9,0),(\alpha^{13},0),(\alpha^{77},0) \}$$
and
$$\{ (\alpha^{73},0),(\alpha^{75},0),(\alpha^{119},0),(\alpha^{26},0) \}~.$$

Thus, a parallelism of lines with 121 disjoint spreads in PG(5,3) is obtained.

\section{Steiner structures}
\label{sec:Steiner}

We suggest to construct a set $\dS$ of 3-dimensional
subspaces of $\F_2^n$, $n$ prime,
in which the cyclic shifts mappings
and the Frobenius mappings form its automorphism group.
The nonzero elements of the field will be represented as one cycle
for this construction. Given a 3-dimensional subspace
$$\{ {\bf 0} , \alpha^{i_1}, \alpha^{i_2} , \alpha^{i_3},
\alpha^{i_4} , \alpha^{i_5} , \alpha^{i_6} , \alpha^{i_7} \}$$
in $\dS$, we require that for each $0 \leq \ell \leq n-1$
and $0 \leq j \leq 2^n-2$,
$$\{ {\bf 0} , \alpha^{i_1 2^\ell +j}, \alpha^{i_2 2^\ell +j} , \alpha^{i_3 2^\ell +j},
\alpha^{i_4 2^\ell +j} , \alpha^{i_5 2^\ell +j} , \alpha^{i_6 2^\ell +j} , \alpha^{i_7 2^\ell +j} \}$$
will be also a 3-dimensional subspace of $\dS$. In other words, $X \in \dS$ if and only if
${\Phi_j (  \Upsilon_\ell (X)) \in \dS}$, for every  $0 \leq \ell \leq n-1$ and $0 \leq j \leq p^n-2$.

For a given 3-dimensional subspace
$$X=\{ {\bf 0} , \alpha^{i_1}, \alpha^{i_2} , \alpha^{i_3},
\alpha^{i_4} , \alpha^{i_5} , \alpha^{i_6} , \alpha^{i_7} \}$$
of $\F_2^n$, let the \emph{difference set} of $X$, $\Delta (X)$,
be the set of integers defined by
$$
\Delta (X) \deff \{ i_r - i_s ~:~  1 \leq r,~s \leq 7,~ r \neq s \}~.
$$

\begin{lemma}
If $X$ is a 3-dimensional subspace of $\F_2^n$, where $2^n-1 \not\equiv 0~(\bmod~7)$,
then the cyclic shifts of $X$ form $2^n-1$ distinct 3-dimensional subspaces.
\end{lemma}

A 3-dimensional subspace $X$ of $\F_2^n$ will be called \emph{complete}
if $| \Delta (X)| =42$. Two complete 3-dimensional subspaces $X,~Y$ of $\F_2^n$
will be called \emph{disjoint complete} if $\Delta (X) \cap \Delta (Y) = \varnothing$.
Each 3-dimensional subspace $X$ of $\F_2^n$ contains seven two-dimensional
subspaces of $\F_2^n$. If $| \Delta (X)| =42$ then no two of these seven
seven two-dimensional subspaces are cyclic shifts of each other.

\begin{lemma}
If $X$ is a 3-dimensional subspace of $\F_2^n$ and $| \Delta (X)| =42$
then the cyclic shifts of $X$ form $2^n-1$ distinct 3-dimensional subspaces.
The $7 \cdot (2^n-1)$ two-dimensional subspaces of these
$2^n-1$ distinct 3-dimensional subspaces are all distinct.
\end{lemma}

\begin{theorem}
If $n \equiv 1~(\bmod~6)$ and
there exist $\frac{2^n-2}{42}$ pairwise disjoint complete
3-dimensional subspaces then there exists a Steiner structure
$\dS_2[2,3,n]$.
\end{theorem}
For $n=7$ there exist only two pairwise disjoint complete
3-dimensional subspaces of $\F_2^7$ instead of three needed for a Steiner
structure $\dS_2[2,3,7]$. If $n=13$ then $\frac{2^{13}-2}{42} =195$
and the search for 195 pairwise disjoint complete
3-dimensional subspaces of $\F_2^{13}$ looks to be out of reach.
Therefore, we try to increase the size of the automorphism group
by adding the Frobenius mappings into the equation.

For a given 3-dimensional subspace
$$X=\{ {\bf 0} , \alpha^{i_1}, \alpha^{i_2} , \alpha^{i_3},
\alpha^{i_4} , \alpha^{i_5} , \alpha^{i_6} , \alpha^{i_7} \}$$
of $\F_2^n$ let the \emph{coset difference set} of $X$, $C(\Delta (X))$,
be the set of integers defined by
$$
C(\Delta (X)) \deff \{ C(i_r - i_s) ~:~  1 \leq r,~s \leq 7,~ r \neq s \}~.
$$

A 3-dimensional subspace $X$ of $\F_2^n$ will be called \emph{coset complete}
if $| C(\Delta (X)| =42$. Two coset complete 3-dimensional subspaces $X,~Y$ of $\F_2^n$
will be called \emph{disjoint coset complete} if $C(\Delta (X)) \cap C(\Delta (Y)) = \varnothing$.

\begin{theorem}
If $n \equiv 1~(\bmod~6)$ is a prime and
there exist $\frac{2^n-2}{42n}$ pairwise disjoint coset complete
3-dimensional subspaces then there exists a Steiner structure
$\dS_2[2,3,n]$.
\end{theorem}

A search for pairwise disjoint
coset complete 3-dimensional subspaces of $\F_2^n$ is done as follows.
$\{ {\bf 0}, \alpha^{i_1} , \alpha^{i_2}, \alpha^{i_3} \}$
is a two-dimensional subspace if and only if
$\alpha^{i_1} + \alpha^{i_2} + \alpha^{i_3}=0$. Also,
$\{ {\bf 0}, \alpha^{i_1} , \alpha^{i_2}, \alpha^{i_3} \}$
is a two-dimensional subspace if and only if
$\{ {\bf 0}, \alpha^{i_1+j} , \alpha^{i_2+j}, \alpha^{i_3+j} \}$
is a two-dimensional subspace for every integer $j$. Therefore,
$C(i_2-i_1)$, $C(i_1-i_2)$, $C(i_3-i_1)$, $C(i_1-i_3)$,
$C(i_3-i_2)$, $C(i_2-i_3)$, always appear together in
a coset difference set. It follows that we can partition the
$\frac{2^n-2}{n}$ cyclotomic cosets of size $n$,
into $\frac{2^n-2}{6n}$ groups of cosets and instead of 42 integers in a coset
difference set we should consider only 7 such integers.
We form a graph $G(V,E)$ as follows. The set $V$ of vertices
for $G$ are represented by 7-subsets of the set of $\frac{2^n-2}{6n}$ elements
which represents the $\frac{2^n-2}{6n}$ groups of cosets. Such a 7-subset
$v$ represents a vertex if and only if there exists a 3-dimensional
subspace $X$ of $\F_2^n$ whose coset difference set $C(\Delta(X))$
is represented by $v$. For two vertices
$v_1,v_2 \in V$ represented by 7-subsets, there is an edge
$\{ v_1 , v_2 \} \in E$ if and only if $v_1 \cap v_2 = \varnothing$,
i.e. the related 3-dimensional subspaces are disjoint coset complete.
A clique with $m$ vertices in $G$ represents $m$ pairwise disjoint
coset complete 3-dimensional subspaces. A clique with
$\frac{2^n-2}{42n}$ vertices represents a Steiner structure
$\dS_2 [2,3,n]$.

A program which generates this graph for $n=13$ was written.
For $n=13$ we have $\frac{2^{13}-2}{13} =630$ cyclotomic
cosets of size 13, and 105 groups of cosets.
A clique of size $\frac{2^n-2}{42n} =15$ in this graph represents a Steiner
structure $\dS_2[2,3,13]$. Unfortunately, it is not feasible
to check even a small fraction of the subsets with 15 vertices
of this graph. Cliques of size 14 in the graph were found in a rate
of more than 10 such cliques per minute. Such a clique represents
a code with 1490762 3-dimensional subspaces of $\F_2^7$ and minimum
distance 4 (compared to the largest code of size 1221296 known
before~\cite{EtSi11}, and the largest code of size 1154931 in
which the cyclic shifts form the automorphism group~\cite{KoKu}).

\section{Conclusion and future work}
\label{sec:conclusion}

This work is a preliminary attempt to solve two of the
most intriguing and difficult problems connecting
design theory and coding theory in the Grassmannian (and the
related projective geometry),
namely the existence of nontrivial Steiner structure
and construction of parallelism of lines in the projective
space.
A parallelism of lines in PG(5,3) was given and an idea
for a general construction of a parallelism in PG($n,p$)
was suggested.
A Steiner structure $\dS_2[2,3,13]$ was not constructed,
but a related code which is the largest known
one for the related parameters, was found and the
foundations to find such structure were suggested.
We would like to see the first nontrivial
Steiner structure and a new infinite family of parallelisms
with new parameters based on the foundation given in this work.

%%%%%%%%%%%%%%%%%%%%%%%%%%%%%%%%%%%%%%%%%%%%%%%%%%%%%%%%%%%%%%%%%%%%%%
%%%%%%%%%%%%%%%%%%%%%%%%%%%%%%%%%%%%%%%%%%%%%%%%%%%%%%%%%%%%%%%%%%%%%%
%%%%%%%%%%%%%%%%%%%%%%%%%%%%%%%%%%%%%%%%%%%%%%%%%%%%%%%%%%%%%%%%%%%%%%
%%%%%%%%%%%%%%%%%%%%%%%%%%%%%%%%%%%%%%%%%%%%%%%%%%%%%%%%%%%%%%%%%%%%%%

%\bibliography{allbib}

\end{document}